\documentclass{article}
\usepackage{graphicx}
\usepackage{amsmath}
\usepackage{amsthm}
\usepackage{caption}
\providecommand{\keywords}[1]{\textbf{Keywords} #1}
\newtheorem{remark}{Remark}

\newtheorem{theorem}{Theorem}

\begin{document}

\title{Stabilized subgrid multiscale finite element formulation for advection-diffusion-reaction equation with variable coefficients coupled with Stokes-Darcy equation}

\author{ Manisha Chowdhury  ,  B.V. Rathish Kumar \thanks{ Email addresses: chowdhurymanisha8@gmail.com (M. Chowdhury) and drbvrk11@gmail.com (B.V.R. Kumar)  } }
      
\date{Indian Institute of Technology Kanpur \\ Kanpur, Uttar Pradesh, India}

\maketitle
\begin{abstract}
In this paper subgrid multiscale stabilized finite element method for Advection-Diffusion-Reaction (ADR) equation coupled with Stokes-Darcy flow problem has been studied. Here the advection velocity involved in ADR equation obeys Stokes-Darcy flow equation. In this study the approach of algebraic approximation of stabilization parameter has been considered. Further apriori error estimation has been elaborately carried out.
\end{abstract}

 

\section{Introduction}
Transport phenomena, mathematically expressed in terms of Advection-Diffusion-Reaction(ADR) equation has always been an active area of research in the fields of Biomedical Engineering, Environmental Sciences, Chemical Engineering etc. Our previous works \cite{RefE},\cite{RefI} have focused on transport equation with spatially variable diffusion and advection coefficients. In this paper we have considered only diffusion coefficients as variable and the advection velocity comes from Stokes-Darcy flow problem. There are few studies \cite{RefA}, \cite{RefC} available in literature on coupled Stokes-Darcy\ transport equation, but none of them has studied stabilized subgrid multiscale finite element method for both the equations. Being most general stabilization method now-a-days subgrid method has been considered more suitable finite element method to be dealt with. It involves split of the unknown true solution is chosen to be the standard Galerkin finite element solution and obtaining the unresolvable solution in terms of the known resolvable one, we will finally arrive at the subgrid formulation. \cite{RefE} has derived an expression of the stabilization parameter for subgrid formulation of ADR equation with variable coefficients and another study \cite{RefF} has  found out the same for Stokes-Darcy equation. Both these studies have followed the approach of algebraic approximation of the parameter. Combining those results we will have the required stabilization parameter for this coupled problem. Here further we have derived apriori error estimation the coupled problem.  \vspace{1mm} \\
The paper is organised as: section 2 presents model problem along with weak formulation and subgrid formulation for the weakly coupled problem. In the next section we have carried out apriori error estimation for this stabilized method with the help of interpolation estimates for introducing projection operators. 

\section{Statement of the problem}
 Let $\Omega$ be an open bounded domain in $R^d$, d=2,3. Here we have worked with two dimensional model for the sake of simplicity in further derivation, but it can be extended to three dimensional model straightforward. \vspace{1mm}\\
For an incompressible solvent fluid the Stokes-Darcy (or Brinkman) equation representing its flow is to find $u: \Omega \rightarrow R^2$ and $p: \Omega \rightarrow R$ such that
\begin{equation}
\begin{split}
- \mu \Delta \textbf{u} + \sigma \textbf{u} + \bigtriangledown p & = \textbf{f} \hspace{2mm} in \hspace{2mm} \Omega\\
\bigtriangledown \cdot \textbf{u} &= 0 \hspace{2mm} in \hspace{2mm} \Omega\\
\textbf{u} &= \textbf{0} \hspace{2mm} on \hspace{2mm} \partial\Omega  \\
\end{split}
\end{equation} 
where \textbf{u}= ($u_1,u_2$),p and $\mu$ are velocity, pressure and viscosity of the fluid, $\sigma$ is the inverse of permeability and \textbf{f} is the body force.\vspace{2 mm}\\
And the ADR equation with spatially variable coefficients representing the concentration c of transporting solute in the same domain $\Omega$ along with homogeneous Dirichlet boundary condition is to find c: $\Omega \rightarrow R$ such that,
\begin{equation}
\begin{split}
- \bigtriangledown \cdot \tilde{\bigtriangledown} c + \textbf{u} \cdot \bigtriangledown c + \alpha c & = g \hspace{2mm} in \hspace{2mm} \Omega \\
c &= 0 \hspace{2mm} on \hspace{2mm} \Omega \\
\end{split}
\end{equation}
where the notation, $\tilde{\bigtriangledown}: = (D_1 \frac{\partial}{\partial x}, D_2 \frac{\partial}{\partial y})$ \\
$D_1, D_2$ are spatially variable diffusion coefficients along x-axis and y-axis respectively, $\alpha$ is the reaction coefficient and g denotes the source of solute mass. Here we make some assumptions on the coefficients as follows:\vspace{1mm}\\
(i) $\mu$, $\sigma$ and $\alpha$ are positive constants.\vspace{1mm}\\
(ii) The diffusion coefficients $D_1, D_2$ are continuous functions on bounded domain $\Omega$and hence bounded. Let $D_{1l},D_{2l}$ be upper bounds of them respectively.\vspace{1mm}\\
(iii) The body force $\textbf{f} \in (L^2(\Omega))^2$ and source $g \in L^2(\Omega)$

\subsection{Weak formulation}
Let us consider the standard space $V= H^1_0(\Omega)$ as the admissible space for both velocity fields and concentration and the space $Q=L^2(\Omega)$ for pressure. \vspace{1mm}\\
The weak formulation is to find $(\textbf{u},p) \in V \times V \times Q$ and $c \in V$ such that 
\begin{equation}
\begin{split}
a_S(\textbf{u}, \textbf{v})-b(\textbf{v},p)+ b(\textbf{u},q)& =l_S(\textbf{v}) \hspace{2mm} \forall \textbf{v} \in V \times V , q \in Q \\
a_T(c,d) & = l_T(d) \hspace{2mm} \forall d \in V
\end{split}
\end{equation}
where $a_S(\textbf{u},\textbf{v})= \int_{\Omega} \mu \bigtriangledown \textbf{u}:\bigtriangledown \textbf{v} + \sigma \int_{\Omega} \textbf{u} \cdot \textbf{v}$ and \hspace{1mm} $b(\textbf{v},q)= \int_{\Omega} (\bigtriangledown \cdot \textbf{v}) q$ \vspace{1 mm} \\
 $a_T(c,d) = \int_{\Omega} \tilde{\bigtriangledown}c \cdot \bigtriangledown d + \int_{\Omega} d \textbf{u} \cdot \bigtriangledown c + \alpha\int_{\Omega}cd $ \vspace{1 mm} \\
 $l_S (\textbf{v})= \int_{\Omega} \textbf{f} \cdot \textbf{v}$ and  $l_T(d)= \int_{\Omega} gd$ \vspace{1mm}\\
The stability of the continuous problem (3) has been shown in \cite{RefL} and \cite{RefI}.
\subsection{Subgrid formulation}
Let $\Omega$ be discretized into total $n_{el}$ number of sub-domains $\Omega_k$ for k=1,2,...,$n_{el}$ and $h_k$ be the diameter of each sub-domain respectively. Let h= $\underset{k=1,2,...,n_{el}}{max} h_k$ and $\tilde{\Omega}=\bigcup_{k=1}^{n_el} \Omega_k $ be the union of interior elements. Let $V_h$ and $Q_h$ be the suitable finite dimensional subspaces of $V$ and $Q$ respectively where \vspace{1mm}\\
$V_h= \{ v \in V: v(\Omega_k)= \mathcal{P}^2(\Omega_k)\} $ and  $Q_h= \{ q \in Q_s : q(\Omega_k)= \mathcal{P}^1(\Omega_k)\}$ \vspace{1 mm}\\
where $\mathcal{P}^1(\Omega_k)$ and $\mathcal{P}^2(\Omega_k)$ denote complete polynomial of order 1 and 2 respectively over each $\Omega_k$ for k=1,2,...,$n_{el}$. \vspace{1mm}\\
Now the standard Galerkin finite element formulation of (3) is to find $(\textbf{u}_h,p_h) \in V_h \times V_h\times Q_h$ and $c \in V_h$ such that 
\begin{equation}
\begin{split}
a_S(\textbf{u}_h, \textbf{v}_h)-b(\textbf{v}_h,p_h)+ b(\textbf{u}_h,q_h)& =l_S(\textbf{v}_h) \hspace{2mm} \forall \textbf{v}_h \in V_h \times V_h , q_h \in Q_h \\
a_T(c_h,d_h) & = l_T(d_h) \hspace{2mm} \forall d_h \in V_h
\end{split}
\end{equation}
As we earlier mention that for subgrid formulation the finite element solution will be taken as resolvable scale and we will obtain the subgrid formulation by expressing the unresolvable scale in terms of known solution. Hence the subgrid formulation for (3) is to find $(\textbf{u}_h,p_h) \in V_h \times V_h\times Q_h$ and $c \in V_h$ such that 
\begin{equation}
\begin{split}
a_S(\textbf{u}_h, \textbf{v}_h)-b(\textbf{v}_h,p_h)+ b(\textbf{u}_h,q_h)+ \int_{\Omega'} (-\mathcal{L}_1^* \textbf{V}_h) [\tau] \mathcal{L}_1 \textbf{U}_h   & = l_S(\textbf{v}_h)+\int_{\Omega'} (-\mathcal{L}_1^* \textbf{V}_h) [\tau] \textbf{F}\\
 \forall \hspace{1mm} \textbf{V}_h & = (\textbf{v}_h,q_h) \in V_h \times V_h \times Q_h \\
a_T(c_h,d_h)+ \int_{\Omega'} (-\mathcal{L}_2^* d_h) \tau_3 \mathcal{L}_2 c_h & = l_T(d_h) + \int_{\Omega'} (-\mathcal{L}_2^* d_h) \tau_3 g \\
& \forall \hspace{1mm} d_h \in V_h
\end{split}
\end{equation}
where
\[ \mathcal{L}_1 \textbf{U} =
\begin{bmatrix}
    - \mu(c) \Delta \textbf{u} + \sigma \textbf{u} + \bigtriangledown p\\
    \bigtriangledown \cdot \textbf{u} \\
\end{bmatrix}
,its \hspace{1mm} adjoint \mathcal{L}_1^* \textbf{U}=
\begin{bmatrix}
   - \mu(c) \Delta \textbf{u} + \sigma \textbf{u} - \bigtriangledown p\\
    -\bigtriangledown \cdot \textbf{u} \\     
\end{bmatrix} 
, \textbf{F} =
\begin{bmatrix}
    \textbf{f}      \\
    0      
\end{bmatrix} 
\]
$\mathcal{L}_2 c  = -\bigtriangledown \cdot  \widetilde{\bigtriangledown} c + \overline{u} \cdot \bigtriangledown c + \alpha c $ and its $adjoint$ $\mathcal{L}_2^* c  = -\bigtriangledown \cdot \widetilde{\bigtriangledown} c - \overline{u} \cdot  \bigtriangledown c + \alpha c $\\
the stabilization parameters[],[] are obtained as
\[ [\tau] = diag(\tau_1, \tau_1, \tau_2)=
\begin{bmatrix}
  (c_1 \frac{\mu}{h^2}+  \sigma)^{-1}  &  0 & 0     \\
  0 & (c_1 \frac{\mu}{h^2}+  \sigma)^{-1}   & 0     \\
  0 & 0 & c_2 \mu       
\end{bmatrix} 
and  \hspace{1mm}\tau_3 = (\frac{9D}{4h^2} + \frac{3U}{2h} + \alpha )^{-1}
\]
\section{Error estimation}
In this section we are going to derive apriori estimate in $V$ norm that is in standard $H^1$ norm. Before that we introduce projection operators and carry out error splitting as follows: \vspace{1mm} \\
Let $\textbf{e}=(e_{\textbf{u}},e_p,e_c)$ denote the error where the components are $e_{\textbf{u}}=(e_{u1},e_{u2})= (u_1-u_{1h}, u_2-u_{2h}), e_p= (p-p_h)$ and $e_c=(c-c_h)$. Let us introduce the projection operator corresponding to each component as the following, \vspace{1mm} \\
(i)For any $\textbf{u} \in V \times V $ let there exist an interpolation $P^h_{\textbf{u}}:  V \times V \longrightarrow  V_h \times V_h $ satisfying 
 $b(\textbf{u}-P^h_{\textbf{u}}\textbf{u}, q_h)=0$ \hspace{2mm} $\forall q_h \in Q_h$ 
 and component wise satisfy $L^2$-orthogonality condition i.e. for any $u_i \in V$ \hspace{1mm} $(u_i-P^h_{u_i}u_i,v_h)=0$ \hspace{1mm} $\forall v_h \in V_s^h$ for i=1,2. \vspace{1mm} \\
(ii) Let $P^h_p: Q \longrightarrow Q_h$ be the $L^2$ orthogonal projection given by \\ $\int_{\Omega}(p-P^h_pp)q_h=0$ \hspace{1mm} $\forall p \in Q, \hspace{1mm} \forall q_h \in Q_h$ \vspace{2mm}\\
(iii) Similarly let $P^h_{c}: V \longrightarrow V_h$ satisfy $\int_{\Omega}(c-P^h_c c)v_h=0$ \hspace{1mm} $\forall c \in V, \hspace{1mm} \forall v_h \in V_h$ \vspace{2mm}\\
Now each components of the error can be split into two parts interpolation part, $E^I$ and auxiliary part, $E^A$ as follows: \vspace{1mm}\\
$e_{u1}=(u_1-u_{1h})=(u_1-P^h_{u1}u_1)+(P^h_{u1}u_1-u_{1h})= E^{I}_{u1}+ E^{A}_{u1}$ \vspace{1mm}\\
Similarly, $e_{u2}=E^{I}_{u2}+ E^{A}_{u2}$,
$e_{p}=E^{I}_{p}+ E^{A}_{p}$, and 
$e_{c}=E^{I}_{c}+ E^{A}_{c}$ \vspace{1mm}\\
\textbf{Interpolation estimates \cite{RefJ,RefK}} : for any true solution $u$ with regularity upto (m+1)
\begin{equation}
\|u-P^h_u u\|_l = \|E^I_u\|_l \leq C(m,\Omega) h^{m+1-l} \|v\|_{m+1} 
\end{equation}
where l is a positive integer and C is a constant depending on m and the domain. For l=0 and 1 it  implies standard $L^2(\Omega)$ and $H^1(\Omega)$ norms respectively. We will use $\| \cdot \|$ instead of $\| \cdot \|_0$ to denote $L^2(\Omega)$ norm. \vspace{2mm} \\
Now we carry out apriori error estimate in two parts. In first part we will bound auxiliary error and in the later part using the result obtained in first part we will find  final form of apriori estimate. 

\begin{theorem} \textbf{Auxiliary error estimate}:
For velocity $\textbf{u}_h=(u_{1h},u_{2h})$, pressure $p_h$ and concentration $c_h$ belonging to $V_h \times V_h , Q_h , V_h$ respectively satisfying (5), assume sufficient regularity of exact solution in equations (1)-(2). Then there exists constants $C_1$ and $C_2$, depending upon $\textbf{u}$,p and c respectively, such that
\begin{equation}
\begin{split}
\|E^A_{u1}\|^2_V + \|E^A_{u2}\|^2_V+ \|E^A_p\|_Q^2  & \leq C_1(\textbf{u},p) h^2\\
\|E^A_c\|_V^2 & \leq C_2(c)h^2
\end{split}
\end{equation}
\end{theorem}

\begin{proof}
We divide the proof into two parts: first part contains estimation of auxiliary error for velocity and pressure components separately and in the later part we will find bound for auxiliary error corresponding to concentration.
Now for true solution (5) becomes
\begin{equation}
\begin{split}
a_S(\textbf{u}, \textbf{v}_h)-b(\textbf{v}_h,p)+ b(\textbf{u},q_h)& =l_S(\textbf{v}_h) \hspace{2mm} \forall \textbf{v}_h \in V_h \times V_h , q_h \in Q_h \\
a_T(c,d_h) & = l_T(d_h) \hspace{2mm} \forall d_h \in V_h
\end{split}
\end{equation}
\textbf{First part:} Subtracting first equation of (8) from that of (5) we will have, $\forall \textbf{v}_h \in V_h \times V_h , q_h \in Q_h $
\begin{multline}
a_S(\textbf{u}-\textbf{u}_h, \textbf{v}_h)-b(\textbf{v}_h,p-p_h)+ b(\textbf{u}-\textbf{u}_h,q_h)= \int_{\Omega'} (\mathcal{L}_1^* \textbf{V}_h) [\tau] (\textbf{F}- \mathcal{L}_1 \textbf{U}_h )
\end{multline}
After explicitly writing the bilinear forms and using the error splitting we have the above equation as follows
\begin{multline}
\mu \int_{\Omega} \bigtriangledown E^I_{\textbf{u}} : \bigtriangledown \textbf{v}_h + \mu \int_{\Omega} \bigtriangledown E^A_{\textbf{u}} : \bigtriangledown \textbf{v}_h + \sigma \int_{\Omega} E^I_{\textbf{u}} \cdot \textbf{v}_h + \sigma \int_{\Omega} E^A_{\textbf{u}} \cdot \textbf{v}_h - \int_{\Omega}(\bigtriangledown \cdot \textbf{v}_h) E^I_p - \\
\int_{\Omega} (\bigtriangledown \cdot \textbf{v}_h) E^A_p+ \int_{\Omega} (\bigtriangledown \cdot E^I_{\textbf{u}}) q_h + \int_{\Omega} (\bigtriangledown \cdot E^A_{\textbf{u}}) q_h \\
= \int_{\Omega'} \tau_1 (\mu \Delta v_{1h} - \sigma v_{1h}+ \frac{\partial q_h}{\partial x})(- \mu \Delta (E^I_{u1}+E^A_{u1})+ \sigma (E^I_{u1}+E^A_{u1}) + \frac{\partial E^I_p}{\partial x} + \frac{\partial E^A_p}{\partial x})+ \\
\quad \int_{\Omega'} \tau_1 (\mu \Delta v_{2h} - \sigma v_{2h}+ \frac{\partial q_h}{\partial y})(- \mu \Delta (E^I_{u2}+E^A_{u2})+ \sigma (E^I_{u2}+E^A_{u2}) + \frac{\partial E^I_p}{\partial y} + \frac{\partial E^A_p}{\partial y})+ \\
\quad \int_{\Omega} \tau_2 (\bigtriangledown \cdot \textbf{v}_h)(\bigtriangledown \cdot E^I_{\textbf{u}}) + \int_{\Omega} \tau_2 (\bigtriangledown \cdot \textbf{v}_h)(\bigtriangledown \cdot E^A_{\textbf{u}}) \hspace{40mm}
\end{multline}
Since this holds for all $\textbf{v}_h \in V_h \times V_h$ and $q_h \in Q_h$, replacing them by $E^A_{\textbf{u}}$ and $E^A_p$ respectively  and using properties of projection operators, we finally have
\begin{multline}
\mu \int_{\Omega} \bigtriangledown E^A_{\textbf{u}} : \bigtriangledown E^A_{\textbf{u}} + \sigma \int_{\Omega} E^A_{\textbf{u}} \cdot E^A_{\textbf{u}} = \int_{\Omega'} \tau_1 (\mu \Delta E^A_{u1} - \sigma E^A_{u1} + \frac{\partial E^A_{p}}{\partial x})(- \mu \Delta E^I_{u1}+ \sigma E^I_{u1} +\\
\quad  \frac{\partial E^I_p}{\partial x})+\int_{\Omega'} \tau_1 (\mu \Delta E^A_{u1} - \sigma E^A_{u1}+ \frac{\partial E^A_{p}}{\partial x})(- \mu \Delta E^A_{u1}+ \sigma E^A_{u1} + \frac{\partial E^A_p}{\partial x})+ \int_{\Omega'} \tau_1 (\mu \Delta E^A_{u2} - \sigma E^A_{u2}\\
\quad + \frac{\partial E^A_{p}}{\partial y})(- \mu \Delta E^I_{u2}+ \sigma E^I_{u2} + \frac{\partial E^I_p}{\partial y}) + \int_{\Omega'} \tau_1 (\mu \Delta E^A_{u2} - \sigma E^A_{u2}+ \frac{\partial E^A_{p}}{\partial y})(- \mu \Delta E^A_{u2}+ \sigma E^A_{u2} + \frac{\partial E^A_p}{\partial y})\\
\quad +\int_{\Omega'} \tau_2 (\bigtriangledown \cdot E^A_\textbf{u})(\bigtriangledown \cdot E^I_{\textbf{u}}) + \int_{\Omega'} \tau_2 (\bigtriangledown \cdot E^A_\textbf{u})(\bigtriangledown \cdot E^A_{\textbf{u}})
 - \mu \int_{\Omega} \bigtriangledown E^I_{\textbf{u}} : \bigtriangledown E^A_{\textbf{u}} 
 - \int_{\Omega} (\bigtriangledown \cdot E^I_{\textbf{u}}) E^A_p \\
= S_1 + S_2+ S_3 + S_4+S_5 + S_6 +S_7 + S_8  \hspace{2mm}(say) \hspace{35 mm}
\end{multline}
Now we estimate each of the above terms separately. Before that here we like to make an important observation: by the virtue of chosen finite element spaces it can be clearly said that every element belonging to $V_h$ and $Q_h$ and their derivatives all are bounded functions over each sub-domain $\Omega_k$. Therefore let us consider the positive numbers $M_{1k}, M_{2k},M_{3k},M_{4k},M_{5k},M_{1k}', M_{2k}', M_{3k}', M_{4k}', M_{5k}'$ as bounds for $E^A_{u1}, \Delta E^A_{u1}, \frac{\partial E^A_p}{\partial x}, \frac{\partial E^A_{u1}}{\partial x}, \frac{\partial E^A_{u1}}{\partial y}, E^A_{u2}, \Delta E^A_{u2}, \frac{\partial E^A_p}{\partial y}, \frac{\partial E^A_{u2}}{\partial x}, \frac{\partial E^A_{u2}}{\partial y}$ respectively. \vspace{1mm}\\
To bound $T_1$ simply multiplying the terms, then applying Cauchy-Schwarz inequality over each of them and finally using bounds for $E^As$ we will have
\begin{equation}
\begin{split}
S_1 & \leq \mid \tau_1 \mid \{(\sum_{k=1}^{n_{el}}(\mu^2 M_{2k}+ \sigma \mid \mu \mid M_{1k}+\mid \mu \mid M_{3k})) \| E^I_{u1}\|_2 + (\sum_{k=1}^{n_{el}}(\sigma \mid \mu \mid M_{2k} + \\
& \quad \sigma^2 M_{1k}+ \mid \sigma \mid M_{3k})) \|E^I_{u1}\| + (\sum_{k=1}^{n_{el}}(\mid \mu \mid M_{2k}+ \sigma M_{1k}+ M_{3k}))\|E^I_p\|_1 \} \\
& \leq C \mid \tau_1 \mid \{(\sum_{k=1}^{n_{el}}(\mu^2 M_{2k}+ \sigma \mid \mu \mid M_{1k}+\mid \mu \mid M_{3k})) \| u_1\|_2 + h^2 (\sum_{k=1}^{n_{el}}(\sigma \mid \mu \mid M_{2k}  \\
& \quad + \sigma^2 M_{1k}+ \mid \sigma \mid M_{3k}))  \|u_1\|_2 + (\sum_{k=1}^{n_{el}}(\mid \mu \mid M_{2k}+ \sigma M_{1k}+ M_{3k}))\|p\|_1 \}
\end{split}
\end{equation}
Applying similar arguments we have the following bounds
\begin{equation}
\begin{split}
S_2 & \leq \mid \tau_1 \mid \sum_{k=1}^{n_{el}} (\mu^2 M_{2k}^2 + 2 \sigma \mid \mu \mid M_{1k} M_{2k}+ \sigma^2 M_{1k}^2+ M_{3k}^2)\\
S_3 & \leq C \mid \tau_1 \mid \{(\sum_{k=1}^{n_{el}}(\mu^2 M_{2k}'+ \sigma \mid \mu \mid M_{1k}'+\mid \mu \mid M_{3k}')) \| u_2\|_2 + h^2 (\sum_{k=1}^{n_{el}}(\sigma \mid \mu \mid M_{2k}'  \\
& \quad + \sigma^2 M_{1k}'+ \mid \sigma \mid M_{3k}'))  \|u_2\|_2 + (\sum_{k=1}^{n_{el}}(\mid \mu \mid M_{2k}'+ \sigma M_{1k}'+ M_{3k}'))\|p\|_1 \} \\
S_4 & \leq \mid \tau_1 \mid \sum_{k=1}^{n_{el}}(\mu^2 M_{2k}'^2 + 2 \sigma \mid \mu \mid M_{1k}' M_{2k}'+ \sigma^2 M_{1k}'^2+ M_{3k}'^2)\\
S_5 & \leq \mid \tau_2 \mid C h \sum_{k=1}^{n_{el}}(M_{4k}\|u_1\|_2 + M_{5k}' \|u_2\|_2)\\
S_6 & \leq \mid \tau_2 \mid \sum_{k=1}^{n_{el}}(M_{4k}^2 + M_{5k}'^2) 
\end{split}
\end{equation}
Applying Cauchy-Schwarz and Young's inequality we get
\begin{equation}
\begin{split}
S_7 & = -\mu \int_{\Omega} (\frac{\partial E^I_{u1}}{\partial x} \frac{\partial E^A_{u1}}{\partial x} + \frac{\partial E^I_{u1}}{\partial y} \frac{\partial E^A_{u1}}{\partial y}+ \frac{\partial E^I_{u2}}{\partial x} \frac{\partial E^A_{u2}}{\partial x}+ \frac{\partial E^I_{u2}}{\partial y} \frac{\partial E^A_{u2}}{\partial y}) \\
\end{split}
\end{equation}
\begin{equation}
\begin{split}
& \leq \mid \mu \mid \{\frac{1}{2 \epsilon_1} (\|\frac{\partial E^I_{u1}}{\partial x}\|^2+ \|\frac{\partial E^I_{u1}}{\partial y}\|^2+\|\frac{\partial E^I_{u2}}{\partial x}\|^2+\|\frac{\partial E^I_{u2}}{\partial y}\|^2)+ \frac{\epsilon_1}{2} (\|\frac{\partial E^A_{u1}}{\partial x}\|^2+ \\
& \quad \|\frac{\partial E^A_{u1}}{\partial y}\|^2+ \|\frac{\partial E^A_{u2}}{\partial x}\|^2+ \|\frac{\partial E^A_{u2}}{\partial y}\|^2) \} \\
& \leq \mid \mu \mid \{ C \frac{h^2}{ 2\epsilon_1}  (\|u_1\|_2^2+ \|u_2\|_2^2)+ \frac{\epsilon_1}{2} (\|\frac{\partial E^A_{u1}}{\partial x}\|^2+ \|\frac{\partial E^A_{u1}}{\partial y}\|^2+ \|\frac{\partial E^A_{u2}}{\partial x}\|^2+ \|\frac{\partial E^A_{u2}}{\partial y}\|^2)  \}
\end{split}
\end{equation}
Applying similar argument
\begin{equation}
\begin{split}
S_8  = \int_{\Omega}(\frac{\partial E^I_{u1}}{\partial x }+ \frac{\partial E^I_{u2}}{\partial y }) E^A_p & \leq (\|\frac{\partial E^I_{u1}}{\partial x } \| + \|\frac{\partial E^I_{u2}}{\partial y }\|) \|E^A_p\| \\
& \leq C \frac{h^2}{ \epsilon_2}  (\|u_1\|_2^2+ \|u_2\|_2^2) + \frac{\epsilon_2}{2} \|E^A_p\|^2
\end{split}
\end{equation}
This completes finding bounds for each of the terms in the right hand side of (11). Now we put all the results obtained above into the equation (11) and take out those terms, which are common with the terms of $LHS$, in the $LHS$ from the $RHS$ of the equation. Consequently in the $RHS$ we remain with the terms multiplied by $h^2$ as the stabilization parameters are of order $h^2$. Now we consider $\epsilon_1, \epsilon_2$ in such a manner that all the coefficients in $LHS$ can be made positive. After that taking minimum of all those coefficients and dividing both the sides by that we finally arrive at the completion of the proof as follows:
\begin{equation}
\|E^A_{u1}\|_V^2+ \|E^A_{u2}\|_V^2+ \|E^A_p\|_Q^2
\leq C(\textbf{u},p) h^2
\end{equation}

\textbf{Second part:} Subtracting second equation of (8) from that of (5) we will have, $\forall d_h \in V_h $
\begin{equation}
a_T(c-c_h,d_h)= \int_{\Omega'}(\mathcal{L}_2^* d_h) \tau_3 (g-\mathcal{L}_2 c_h)= \int_{\Omega'}(\mathcal{L}_2^* d_h) \tau_3 \mathcal{L}_2(c- c_h)
\end{equation}
After explicitly writing the bilinear forms and using the error splitting we have the above equation as follows
\begin{multline}
\int_{\Omega} \tilde{\bigtriangledown} E^I_c \cdot \bigtriangledown d_h+ \int_{\Omega} \tilde{\bigtriangledown} E^A_c \cdot \bigtriangledown d_h + \int_{\Omega} d_h \textbf{u} \cdot \bigtriangledown E^I_c + \int_{\Omega} d_h \textbf{u} \cdot \bigtriangledown E^A_c + \alpha\int_{\Omega} E^I_c d_h + \\
\quad \alpha\int_{\Omega} E^A_c d_h= \int_{\Omega'}(-\bigtriangledown \cdot \widetilde{\bigtriangledown} d_h - \overline{u} \cdot  \bigtriangledown d_h + \alpha d_h) \tau_3 ( -\bigtriangledown \cdot  \widetilde{\bigtriangledown} E^I_c + \overline{u} \cdot \bigtriangledown E^I_c + \alpha E^I_c) \\
\int_{\Omega'}(-\bigtriangledown \cdot \widetilde{\bigtriangledown} d_h - \overline{u} \cdot  \bigtriangledown d_h + \alpha d_h) \tau_3 ( -\bigtriangledown \cdot  \widetilde{\bigtriangledown} E^A_c + \overline{u} \cdot \bigtriangledown E^A_c + \alpha E^A_c)
\end{multline}
Since this holds for all $d_h \in V_h $, replacing it by $E^A_c$ in the above equation and using properties of projection operators, we finally have

\begin{multline}
\int_{\Omega} \tilde{\bigtriangledown} E^A_c \cdot \bigtriangledown E^A_c + \alpha \int_{\Omega} E^A_c E^A_c= \int_{\Omega'}(-\bigtriangledown \cdot \widetilde{\bigtriangledown} E^A_c - \overline{u} \cdot  \bigtriangledown E^A_c + \alpha E^A_c) \tau_3 ( -\bigtriangledown \cdot  \widetilde{\bigtriangledown} E^I_c + \\
\quad \overline{u} \cdot \bigtriangledown E^I_c + \alpha E^I_c) +\int_{\Omega'}(-\bigtriangledown \cdot \widetilde{\bigtriangledown} E^A_c - \overline{u} \cdot  \bigtriangledown E^A_c + \alpha E^A_c) \tau_3 ( -\bigtriangledown \cdot  \widetilde{\bigtriangledown} E^A_c + \overline{u} \cdot \bigtriangledown E^A_c + \alpha E^A_c) \\
\quad -\int_{\Omega} \tilde{\bigtriangledown} E^I_c \cdot \bigtriangledown E^A_c- \int_{\Omega} E^A_c \textbf{u} \cdot \bigtriangledown E^I_c - \int_{\Omega} E^A_c \textbf{u} \cdot \bigtriangledown E^A_c \\
=T_1+T_2+T_3+T_4+T_5 \hspace{1mm} (say) \hspace{60mm}
\end{multline}
We will follow the same procedure as the previous part. Under the same observation let us consider the positive numbers $M_{6k}, M_{7k}, M_{8k}, M_{7k}', M_{8k}'$ as bounds for $E^A_c, \frac{\partial E^A_c}{\partial x}, \frac{\partial^2 E^A_c}{\partial x^2}, \frac{\partial E^A_c}{\partial x}, \frac{\partial E^A_c}{\partial y}$ respectively over each sub-domain. \vspace{1mm}\\
First simply multiplying the terms of $T_1$, then applying Cauchy-Schwarz inequality over each of them and finally using bounds for $E^As$ we can find estimate of $T_1$ as follows:
\begin{equation}
\begin{split}
T_1 & \leq \mid \tau_3 \mid (\sum_{k=1}^{n_{el}}(D_{1l} M_{8k}+ D_{2l} M_{8k}' +  D_{u1} M_{7k}+ D_{u2} M_{7k}'+ \alpha M_{6k})) \{( D_{1l}+ D_{2l})  \\
& \quad \|E^I_c\|_2+(\bar{D}_{u1}+\bar{D}_{u2})\|E^I_c\|_1 + \alpha \|E^I_c\|\}\\
& \leq \mid \tau_3 \mid C(\sum_{k=1}^{n_{el}}(D_{1l} M_{8k}+ D_{2l} M_{8k}' + D_{u1} M_{7k}+ D_{u2} M_{7k}'+ \alpha M_{6k})) \{( D_{1l}+ D_{2l}) \\
& \quad  + h (\bar{D}_{u1}+\bar{D}_{u2}) + h^2 \alpha \} \|c\|_2
\end{split}
\end{equation}
Applying similar arguments we have the following bounds
\begin{equation}
\begin{split}
T_2 & \leq \mid \tau_3 \mid \sum_{k=1}^{n_{el}}\{(D_{1l} M_{8k}+ D_{2l} M_{8k}' +  D_{u1} M_{7k}+ D_{u2} M_{7k}'+ \alpha M_{6k})(D_{1l} M_{8k}+ \\
& \quad D_{2l} M_{8k}' +  \bar{D}_{u1} M_{7k}+ \bar{D}_{u2} M_{7k}'+ \alpha M_{6k})\}
\end{split}
\end{equation}
For estimating next terms we are going to use Cauchy-Schwarz and Young's inequalities as follows:
\begin{equation}
\begin{split}
T_3 & = -\int_{\Omega}(D_1 \frac{\partial E^I_c}{\partial x} \frac{\partial E^A_c}{\partial x}+ D_2 \frac{\partial E^I_c}{\partial y} \frac{\partial E^A_c}{\partial y}) \\
& \leq \frac{D_{1l}}{2 \epsilon_3} \|E^I_c\|_1^2 + \frac{D_{2l} \epsilon_3}{2}(\|\frac{\partial E^I_c}{\partial x}\|^2 + \|\frac{\partial E^I_c}{\partial y}\|^2) \\
& \leq  \frac{D_{1l}}{2 \epsilon_3} h^2 \|c\|_2^2+ \frac{D_{2l} \epsilon_3}{2}(\|\frac{\partial E^I_c}{\partial x}\|^2 + \|\frac{\partial E^I_c}{\partial y}\|^2)\\
\end{split}
\end{equation}
\begin{equation}
\begin{split}
T_4 & = - \int_{\Omega}(u_1\frac{\partial E^I_c}{\partial x} + u_2 \frac{\partial E^I_c}{\partial y} ) E^A_c \\
& \leq \frac{(u_{1l}+u_{2l})}{2 \epsilon_4} \| E^I_c\|_1^2 + \frac{(u_{1l}+u_{2l})}{2} \epsilon_4 \|E^A_c\|^2 \\
& \leq \frac{(u_{1l}+u_{2l})}{2 \epsilon_4} h^2 \| c\|_2^2 + \frac{(u_{1l}+u_{2l})}{2} \epsilon_4 \|E^A_c\|^2 
\end{split}
\end{equation}
Similarly
\begin{equation}
\begin{split}
T_5 & = - \int_{\Omega}(u_1\frac{\partial E^A_c}{\partial x} + u_2 \frac{\partial E^A_c}{\partial y} ) E^A_c  \\
& \leq \frac{(u_{1l}+u_{2l})}{2 \epsilon_4}( \| \frac{\partial E^A_c}{\partial x}\|^2 + \|\frac{\partial E^A_c}{\partial y}\|^2)+ \frac{(u_{1l}+u_{2l})}{2} \epsilon_4 \|E^A_c\|^2
\end{split}
\end{equation}
Now combining all the results obtained above into the equation (20) and proceeding as explained at the end of the first part we finally get
\begin{equation}
\|\frac{\partial E^A_c}{\partial x}\|^2+ \|\frac{\partial E^A_c}{\partial y}\|^2 + \|E^A_c\|^2 \leq C_2(c) h^2
\end{equation}
This completes the proof.
\end{proof}

\begin{theorem} \textbf{Apriori error estimate}:
Assuming the same condition as in the previous theorem
\begin{equation}
\begin{split}
\|u_1-u_{1h}\|^2_V + \|u_2-u_{2h}\|^2_V+ \|p-p_h\|_Q^2 + \|c-c_h\|_V^2 & \leq C'(\textbf{u},p,c) h^2\\
\end{split}
\end{equation}
where $C'$ is a constant depending upon $\textbf{u},p,c$.
\end{theorem}
\begin{proof}
We prove this result by first applying triangle inequality and then interpolation estimates and the results obtained in the previous theorem as the following
\begin{multline}
\|u_1-u_{1h}\|^2_V + \|u_2-u_{2h}\|^2_V+ \|p-p_h\|_Q^2 + \|c-c_h\|_V^2 \\
= \|E^I_{u1}+ E^A_{u1}\|_V^2 + \|E^I_{u2}+E^A_{u2}\|_V^2 + \|E^I_c+ E^A_c\|_V^2 + \|E^I_c+E^A_c\|_V^2 \hspace{20 mm} \\
\leq \bar{C}(\|E^I_{u1}\|_V^2+\|E^I_{u2}\|_V^2+\|E^I_{p}\|_V^2+\|E^I_{c}\|_V^2+\|E^A_{u1}\|_V^2+\|E^A_{u2}\|_V^2+\|E^A_{p}\|_V^2+\|E^A_{c}\|_V^2) \\
\leq C'(\textbf{u},p,c) h^2 \hspace{100 mm}
\end{multline}
 This completes the proof.
\end{proof}
\begin{remark}
This error estimation is not computable as it depends upon the true solution, but it implies convergence of the stabilized method and the order of convergence for this method is 2. 
\end{remark}

\section{Numerical Experiment}
In this section through a test case based on hydrological importance \cite{RefM}, we have shown subgrid multiscale stabilization method is performing better than Galerkin method for small diffusion for this weakly coupled Stokes-Darcy/ transport equation. We have considered a source of pollutant dispersed into an incompressible fluid in a simple bounded square domain $\Omega= (0,1) \times (0,1)$. The advection velocity comes from the solution of Stokes-Darcy flow problem and the diffusion coefficients are taken as $D_1= 10^{-7} (1+0.02x)^2$, $D_2= 10^{-8}(1+0.02y)^2$ and the reaction term $\alpha=10$. Table 1 represents the comparison between Galerkin method and SGS method for small diffusion. \vspace{1mm}\\
Again for higher diffusion coefficients compared to the reaction term both the Galerkin method and SGS method perform well. Let us consider the values of diffusion coefficients as $D_1=(1+0.02x)^2$, $D_2=0.1(1+0.02y)^2$ and reaction term $\alpha=0.001$. Table 2 presents the comparison of both Galerkin method and SGS method for diffusion dominated case.
\begin{table}
   \centering
    \begin{tabular}{||c c c c c||}
    \hline 
            & Galerkin Method & Galerkin Method & SGS Method & SGS Method \\
   Mesh size & Error & Order of convergence & Error & Order of convergence\\
    \hline \hline
      10      &  0.00049606        &	&0.000265901                       \\
      \hline
      20      &  0.000152662        & 1.70017 &	6.36986	$e^{-5}$&	2.06155 \\
      \hline
      40      &  3.94719 $e^{-5}$   & 1.95144&	1.68468	$e^{-5}$&	1.91879 \\
      \hline
      80      &  1.43178 $e^{-5}$   & 1.46302&	4.73744	$e^{-6}$&	1.83029  \\
      \hline
      160     &  4.79682 $e^{-6}$   & 1.57766&	5.4902	$e^{-7}$&	1.91276  \\
      \hline
      320     &  1.38697 $e^{-6}$   & 1.79015&	4.45223	$e^{-8}$&	1.89875 \\
      \hline
    \end{tabular}
     \caption{ Comparison between Galerkin method and SGS method for small diffusion}
\end{table}

\begin{table}
   \centering
    \begin{tabular}{||c c c c c||}
    \hline 
            & Galerkin Method & Galerkin Method & SGS Method & SGS Method \\
   Mesh size & Error & Order of convergence & Error & Order of convergence\\
    \hline \hline
      10      &  0.000571444        &	&0.000574645                       \\
      \hline
      20      &  0.000139961        & 2.02959 &	0.000141085	&	2.0261 \\
      \hline
      40      &  3.20151 $e^{-5}$   & 2.1282&	3.22881	$e^{-5}$&	2.12749 \\
      \hline
      80      &  7.74389 $e^{-6}$   & 2.04762&	7.82478	$e^{-6}$&	2.04488  \\
      \hline
      160     &  2.02045 $e^{-6}$   & 1.93838&	2.03639	$e^{-6}$&	1.94204  \\
      \hline
      320     &  4.69837 $e^{-7}$   & 2.10444&	4.75272	$e^{-7}$&	2.09919 \\
      \hline
    \end{tabular}
     \caption{ Comparison between Galerkin method and SGS method for diffusion dominated case}
\end{table}
\begin{remark}
The numerical experiment has proved theoretically established result and shown that the order of convergence under SGS method for both the cases is 2 whereas for small diffusion Galerkin method oscillates and for diffusion dominated flow it behaves well as stabilized method.
\end{remark}

\section{Conclusion} 
In this paper we have studied subgrid multiscale stabilization formulation of weakly coupled Stokes-Darcy and transport equation. As well as we have carried out apriori error estimation, which implies the convergence of the method and finally the numerical experiment verifies the theoretically established result.

\end{document}